\documentclass[12pt]{amsart}

\usepackage{amsmath,bm,amssymb,amsthm,textcomp}
\usepackage{enumitem}
\usepackage{mathtools}
\usepackage{hyperref}
\usepackage[numbers,sort&compress]{natbib}

\DeclareMathOperator{\hdim}{\dim_H}

\DeclareMathOperator{\realizable}{re}
\DeclareMathOperator{\identity}{id}

\theoremstyle{plain}
\newtheorem{theorem}{Theorem}[section]
\newtheorem*{notheorem}{Theorem}
\newtheorem*{maintheorem}{Main Theorem}
\newtheorem{proposition}[theorem]{Proposition}
\newtheorem{claim}{Claim}

\theoremstyle{definition}

\theoremstyle{remark}
\newtheorem{remark}[theorem]{Remark}

\counterwithin{equation}{section}

\allowdisplaybreaks

\begin{document}

\title[Exceptional set of the law of large numbers in Pierce expansions]{An elementary proof that the set of exceptions to the law of large numbers in Pierce expansions has full Hausdorff dimension}

\author{Min Woong Ahn}
\address{Department of Mathematics Education, Silla University, 140, Baegyang-daero 700beon-gil, Sasang-gu, Busan, 46958, Republic of Korea}
\email{minwoong@silla.ac.kr}

\date{\today}


\begin{abstract}
The digits of the Pierce expansion satisfy the law of large numbers. It is known that the Hausdorff dimension of the set of exceptions to the law of large numbers is $1$. We provide an elementary proof of this fact by adapting Jun Wu's method, which was originally used for Engel expansions. Our approach emphasizes the fractal nature of exceptional sets and avoids advanced machinery, thereby relying instead on explicit sequences and constructive techniques. Furthermore, our method opens the possibility of extending similar analyses to other real number representation systems, such as the Engel, L{\"u}roth, and Sylvester expansions, thus paving the way for further explorations in metric number theory and fractal geometry.
\end{abstract}

\maketitle

\tableofcontents

\section{Introduction} \label{introduction}

For $x \in [0,1]$, we denote the {\em Pierce expansion} or {\em alternating Engel expansion} of $x$ by $[d_1(x), d_2(x), \dotsc]_P$, i.e., 
\[
x = [d_1(x), d_2(x), \dotsc]_P \coloneqq \sum_{j=1}^\infty \left( (-1)^{j+1} \prod_{k=1}^j \frac{1}{d_k(x)} \right),
\]
where $( d_k(x) )_{k \in \mathbb{N}}$ is either 
\begin{enumerate}[label=\upshape(\roman*), ref=\roman*, leftmargin=*, widest=ii]
\item
a sequence of positive integers that satisfy $d_{k+1} (x) > d_k(x)$ for all $k \in \mathbb{N}$ or 
\item
a sequence of strictly increasing initial $n$ positive integer-valued terms followed by $\infty$ for the rest for some $n \in \mathbb{N}$,
\end{enumerate}
with the conventions $\infty \cdot \infty = \infty$, $1/\infty = 0$, and $c \cdot \infty = \infty$ for any $c>0$ (see \cite{Ahn23a} and \cite{Sha86}). The $d_k(x)$ are called the {\em digits} of the Pierce expansion of $x$, and the digits are dynamically obtained in the following way. Consider two maps $d_1 \colon [0,1] \to \mathbb{N} \cup \{ \infty \}$ and $T \colon [0,1] \to [0,1]$ defined by the following:
\begin{align} \label{Pierce algorithm 1}
d_1(x) \coloneqq \begin{cases} \lfloor 1/x \rfloor, &\text{if } x \neq 0; \\ \infty, &\text{if } x =0, \end{cases} 
\quad \text{and} \quad 
T(x) \coloneqq \begin{cases} 1 - d_1 (x) x, &\text{if } x \neq 0; \\ 0, &\text{if } x =0, \end{cases}
\end{align}
respectively, where $\lfloor \xi \rfloor$ denotes the largest integer not exceeding $\xi \in \mathbb{R}$. Then, the sequence of digits are recursively defined by the following:
\begin{align} \label{Pierce algorithm 2}
d_k(x) \coloneqq d_1 (T^{k-1}(x)) \quad \text{for each } k \in \mathbb{N}.
\end{align}
If $n \coloneqq \inf \{ j \in \mathbb{N} \cup \{ 0 \} : T^j (x) = 0 \}$ is finite, then $d_1(x) < \dotsb < d_n(x) < \infty$ and $d_k(x) = \infty$ for all $k \geq n+1$; furthermore, if $n \geq 2$, then $d_{n-1}(x) + 1 < d_n(x)$ (see \cite[Proposition 2.1]{Ahn23a} and \cite[pp.~23--24]{Sha86}).

In his 1986 paper, Shallit established the {\em law of large numbers (LLN)} of Pierce expansion digits \cite[Theorem~16]{Sha86} : {\em For Lebesgue-almost every $x \in [0,1]$, we have the following:}
\begin{align} \label{convergence law}
\lim_{n \to \infty} (d_n(x))^{1/n} = e.
\end{align}
One natural question that arises is to determine the Hausdorff dimension of the exceptional set where \eqref{convergence law} fails. This is because, although the LLN tells us that the exceptional set has null Lebesgue measure, it might be possible that the exceptional set is fairly large in the Hausdorff dimension sense. 

In the context of Engel expansions, which are non-alternating variant of Pierce expansions, Galambos \cite{Gal76} posed a similar question concerning the Hausdorff dimension of the set on which the LLN for the Engel expansions fails. (See \cite{ERS58} for the LLN in Engel expansions). Wu \cite{Wu00} resolved this question through an ingenious construction. Specifically, he first fixed a subset $A \subseteq (0,1]$ of positive Lebesgue measure (and hence of full Hausdorff dimension); then, he defined a sequence of maps $(g_n \colon A \to (0,1])_{n \geq M}$ for some sufficiently large $M \in \mathbb{N}$. Then, he demonstrated that the LLN fails on $g_n(A)$ and that $g_n^{-1} \colon g_n(A) \to A$ is $(1/\alpha_n)$-H{\"o}lder continuous for all $n \geq M$, where $\alpha_n \to 1^+$ as $n \to \infty$. This allowed him to conclude that the Hausdorff dimension of $g_n(A)$ is at least $\alpha_n$, thereby establishing that the exceptional set has full Hausdorff dimension.

For Pierce expansions, we recently established the following two results in \cite{Ahn24} concerning the set of exceptions to the LLN \eqref{convergence law}, the second of which serves as the main theorem of this paper.

\begin{notheorem} {\rm\cite[Corollary~1.3]{Ahn24}} 
For each $\alpha \in [1, +\infty]$, the set 
\[
\left\{ x \in [0,1] : \lim_{n \to \infty} (d_n(x))^{1/n} = \alpha \right\}
\]
has full Hausdorff dimension.
\end{notheorem}

\begin{maintheorem} {\rm\cite[Corollary~1.4]{Ahn24}}
The subset of $[0,1]$ on which \eqref{convergence law} fails has full Hausdorff dimension.
\end{maintheorem}

In fact, \cite{Ahn24} established even more general results. The primary tools used in our proofs are classical in fractal geometry. For instance, we defined a suitable symbolic space and applied the so-called {\em mass distribution principle} (see \cite[Example 4.6(a)]{Fal14}) to derive a lower bound for the Hausdorff dimension. In certain cases, we also obtained an upper bound using the lower box-counting dimension and the fact that the Hausdorff dimension does not exceed the lower box-counting dimension (see \cite[Propositions~3.4 and 4.1]{Fal14}). For further details, see \cite{Ahn24}.

In this paper, we provide a more elementary proof of our recent result that the Hausdorff dimension of the set where \eqref{convergence law} fails is $1$ (Main Theorem). Our approach modifies Wu’s \cite{Wu00} method, which was originally developed for Engel expansions. We consider our method elementary because it avoids advanced machinery, such as abstract measure-theoretic constructions, and instead relies on explicit sequences and constructive techniques. This contrasts with previous proofs that employed symbolic dynamics and dimension theory.

Wu's proof for Engel expansions \cite{Wu00} involved intricate mappings and H{\"o}lder continuity arguments to show that the exceptional set for the LLN has full Hausdorff dimension. While effective, his approach does not directly transfer to Pierce expansions due to differing digit constraints (see Remark~\ref{digit condition remark} for the details). Our method adapts and simplifies Wu's framework by extensively using properties unique to Pierce expansions, such as the strict increasing condition for digits, as stated in the undermentioned Proposition \ref{digit condition lemma}. This simplification contributes to the accessibility of the proof, thus offering an alternative perspective that is educationally beneficial and computationally straightforward.

Throughout this paper, we denote the Hausdorff dimension by $\hdim$ and the Lebesgue measure on $[0,1]$ by $\mathcal{L}$. The set of positive integers is denoted by $\mathbb{N}$, the set of non-negative integers by $\mathbb{N}_0$, the set of extended positive integers by $\mathbb{N}_\infty \coloneqq \mathbb{N} \cup \{ \infty \}$, and the set of irrational numbers in $[0,1]$ by $\mathbb{I}$. Following standard convention, we define $\infty \cdot \infty \coloneqq \infty$, $c \cdot \infty \coloneqq \infty$ for any $c > 0$, and ${1}/{\infty} \coloneqq 0$. For any map $g \colon X \to Y$ and a subset $A \subseteq X$, we write $g|_A$ to denote the restriction of $g$ to $A$.

\section{Preliminaries} \label{preliminaries}

In this section, we introduce some elementary facts on Pierce expansions and the Hausdorff dimension, which will be used in the proof of the Main Theorem.

We begin by listing some well-known properties of Pierce expansions.

\begin{proposition} {\rm\cite[Proposition 2.1]{Ahn23a}} \label{digit condition lemma}
For each $x \in [0,1]$, the following hold for all $n \in \mathbb{N}$:
\begin{enumerate}[label=\upshape(\roman*), ref=\roman*, leftmargin=*, widest=ii]
\item \label{digit condition lemma 1}
$d_{n+1}(x) \geq d_n(x)+1$; and
\item \label{digit condition lemma 2}
$d_n(x) \geq n$.
\end{enumerate}
\end{proposition}

\begin{proof}
Both parts follow from \eqref{Pierce algorithm 1} and \eqref{Pierce algorithm 2}. See \cite{Ahn23a} for the details.
\end{proof}

We introduce sets of sequences that form subsets of $\mathbb{N}_\infty^\mathbb{N}$ and are closely related to the Pierce expansion digit sequences. These sets were studied in detail in \cite{Ahn23a}.

Define $\Sigma_0$ as the singleton set:
\[
\Sigma_0 \coloneqq \{ (\infty, \infty, \dotsc) \}.
\]
For each $n \in \mathbb{N}$, define the following:
\[
\Sigma_n \coloneqq \{ (\sigma_k)_{k \in \mathbb{N}} \in \mathbb{N}^{\{ 1, \dotsc, n \}} \times \{ \infty \}^{\mathbb{N} \setminus \{ 1, \dotsc, n \}} : \sigma_1 < \sigma_2 < \dotsb < \sigma_n \}.
\]
Next, define the following:
\[
\Sigma_\infty \coloneqq \{ (\sigma_k)_{k \in \mathbb{N}} \in \mathbb{N}^{\mathbb{N}} : \sigma_k < \sigma_{k+1} \text{ for all } k \in \mathbb{N} \}.
\]
Finally, define the full sequence set as follows:
\[
\Sigma \coloneqq \Sigma_0 \cup \bigcup_{n \in \mathbb{N}} \Sigma_n \cup \Sigma_\infty.
\]

We say that a sequence $\sigma \coloneqq (\sigma_k)_{k \in \mathbb{N}}$ in $\mathbb{N}_\infty^{\mathbb{N}}$ is {\em Pierce realizable} if there exists $x \in [0,1]$ such that $d_k(x) = \sigma_k$ for all $k \in \mathbb{N}$. We denote the collection of all Pierce realizable sequences by $\Sigma_{\realizable}$.

\begin{proposition} {\rm\cite[Proposition 2.4]{Ahn23a}} \label{strict increasing condition}
Let $\sigma \coloneqq (\sigma_k)_{k \in \mathbb{N}} \in \mathbb{N}_\infty^{\mathbb{N}}$. Then, $\sigma \in \Sigma_{\realizable}$ if and only if one of the following holds:
\begin{enumerate}[label=\upshape(\roman*), ref=\roman*, leftmargin=*, widest=iii]
\item
$\sigma \in \Sigma_0 \cup \Sigma_1$;

\item
$\sigma \in \Sigma_n$ for some $n \geq 2$, with $\sigma_{n-1} + 1 < \sigma_n$; and

\item
$\sigma \in \Sigma_\infty$.
\end{enumerate}
\end{proposition}

Define $f \colon [0,1] \to \Sigma$ by 
\[
f(x) \coloneqq (d_k(x))_{k \in \mathbb{N}}
\]
for each $x \in [0,1]$. Define $\varphi \colon \Sigma \to [0,1]$ by
\begin{align*} 
\varphi (\sigma) \coloneqq \sum_{j=1}^\infty \left( (-1)^{j+1} \prod_{k=1}^j \frac{1}{\sigma_k} \right) = \frac{1}{\sigma_1} - \frac{1}{\sigma_1 \sigma_2} + \dotsb + \frac{(-1)^{n+1}}{\sigma_1 \dotsm \sigma_n} + \dotsb
\end{align*}
for each $\sigma \coloneqq (\sigma_k)_{k \in \mathbb{N}} \in \Sigma$. 

\begin{proposition} {\rm (See \cite[Section 3.2]{Ahn23a})} \label{identity lemma}
We have $\varphi \circ f = \identity_{[0,1]}$ and $f \circ (\varphi|_{\Sigma_{\realizable}}) = \identity_{\Sigma_{\realizable}}$. In particular, $f|_{\mathbb{I}} \colon \mathbb{I} \to \Sigma_\infty$ and $\varphi|_{\Sigma_\infty} \colon \Sigma_\infty \to \mathbb{I}$ are inverses to each other.
\end{proposition}

\begin{remark} \label{remark for rationals}
In view of Proposition \ref{identity lemma}, we remark that for every rational number $x \in [0,1]$, there exists $n = n(x) \in \mathbb{N}$ such that $d_k(x) = \infty$ for all $k \geq n$. Consequently, we have $(d_k(x))^{1/k} = \infty$ for all $k \geq n$, which is consistent with the convention $\infty^k = \infty$ for any $k \in \mathbb{N}$. This tells us that the convergence in \eqref{convergence law} does not hold for any rational number.
\end{remark}

Let $n \in \mathbb{N}$ and $\sigma \coloneqq (\sigma_k)_{k \in \mathbb{N}} \in \Sigma_n$. Define the {\em fundamental interval} associated with $\sigma$ by the following:
\[
I (\sigma) \coloneqq \{  x \in [0,1] : d_k(x) = \sigma_k \text{ for all } 1 \leq k \leq n \}.
\]
Next, we define a sequence $\widehat{\sigma} \coloneqq (\widehat{\sigma}_k)_{k \in \mathbb{N}} \in \Sigma_n$ by setting $\widehat{\sigma}_n \coloneqq \sigma_n + 1$ and $\widehat{\sigma}_k \coloneqq \sigma_k$ for $k \in \mathbb{N} \setminus \{ n \}$, i.e.,
\[
\widehat{\sigma} = (\sigma_1, \dotsc, \sigma_{n-1}, \sigma_n+1, \infty, \infty, \dotsc).
\]

\begin{proposition} {\rm(See \cite[Theorem 1]{Sha86})} and {\cite[Proposition 3.5]{Ahn23a}} \label{I sigma}
For each $n \in \mathbb{N}$ and $\sigma \in \Sigma_n$, we have
\[
I (\sigma) = \begin{cases}
(\varphi (\widehat{\sigma}), \varphi (\sigma)], &\text{if $n$ is odd;} \\
[\varphi (\sigma), \varphi (\widehat{\sigma})), &\text{if $n$ is even,}
\end{cases}
\quad \text{or} \quad
I (\sigma) = \begin{cases}
(\varphi (\widehat{\sigma}), \varphi (\sigma)), &\text{if $n$ is odd;} \\
(\varphi (\sigma), \varphi (\widehat{\sigma})), &\text{if $n$ is even,}
\end{cases}
\]
according as $\sigma \in \Sigma_{\realizable}$ or $\sigma \not \in \Sigma_{\realizable}$.
\end{proposition}

Let $\sigma \coloneqq (\sigma_k)_{k \in \mathbb{N}} \in \Sigma$. Define $\sigma^{(0)} \coloneqq (\infty, \infty, \dotsc)$. For each $n \in \mathbb{N}$, we define $\sigma^{(n)} \coloneqq (\tau_k)_{k \in \mathbb{N}} \in \Sigma$ by setting $\tau_k \coloneqq \sigma_k$ for $1 \leq k \leq n$ and $\tau_k \coloneqq \infty$ for $k \geq n+1$, i.e.,
\[
\sigma^{(n)} = (\sigma_1, \dotsc, \sigma_n, \infty, \infty, \dotsc).
\]

\begin{proposition} {\rm\cite[Lemma 3.29]{Ahn23a}} \label{distance between sigma and sigma n}
Let $\sigma \coloneqq (\sigma_k)_{k \in \mathbb{N}} \in \Sigma$. For each $n \in \mathbb{N}$, we have the following:
\[
| \varphi (\sigma) - \varphi (\sigma^{(n)}) | \leq \prod_{k=1}^{n+1} \frac{1}{\sigma_k}.
\]
\end{proposition}

The following proposition will be crucial in establishing a lower bound for the Hausdorff dimension.

\begin{proposition} {\rm\cite[Proposition 3.3(a)]{Fal14}} \label{Holder dimension}
Let $F \subseteq \mathbb{R}$, and suppose that $g \colon F \to \mathbb{R}$ is $\alpha$-H{\"o}lder continuous, i.e., there exist constants $\alpha>0$ and $c>0$ such that
\[
|g (x) - g (y)| \leq c |x-y|^\alpha
\]
for all $x, y \in F$. Then, $\hdim g (F) \leq (1/\alpha) \hdim F$.
\end{proposition}

\section{Proof of the result} \label{proof of the result}

It is worth reiterating that our proof idea is inspired by \cite{Wu00}, in which the exceptional set of the LLN in Engel expansions was discussed.

\begin{proof}[Proof of Main Theorem]
Put 
\[
E \coloneqq \{ x \in [0,1] : (d_n(x))^{1/n} \not \to e \text{ as } n \to \infty \} = \{ x \in [0,1] : \text{\eqref{convergence law} fails} \}. 
\]
The inequality $\hdim E \leq 1$ is obvious; therefore, it suffices to prove the reverse inequality. Our aim is to find a sequence of subsets of $E$ such that the Hausdorff dimensions of the subsets can be made arbitrarily close to $1$.

Recall that the LLN states that the limit in \eqref{convergence law} holds $\mathcal{L}$-almost everywhere in $[0,1]$. By Egorov's theorem, there exists a Lebesgue measurable subset $A \subseteq [0,1]$ with $\mathcal{L} (A) > 0$ such that $(d_n (x))^{1/n} \to e$ as $n \to \infty$ uniformly on $A$. As noted in Remark \ref{remark for rationals}, we have $A \subseteq \mathbb{I}$, so that $f(A) \subseteq f(\mathbb{I}) = \Sigma_\infty$, where the equality follows from Proposition \ref{identity lemma}. Since $2<e<4$, by the uniform convergence, we can find an $N \in \mathbb{N}$ such that
\begin{align} \label{bounds for dn}
2^n < d_n (x) < 4^n \quad \text{for all $n \geq N$ and for all $x \in A$}.
\end{align}

\begin{claim} \label{dimension of A} 
We have $\hdim A = 1$.
\end{claim}

\begin{proof}[Proof of Claim \ref{dimension of A}] \renewcommand\qedsymbol{$\blacksquare$}
The claim is clear since any subset of positive Lebesgue measure in $[0,1]$ has full Hausdorff dimension (see \cite[p.~45]{Fal14}).
\end{proof}

Let $M \in \mathbb{N}$ be arbitrary such that $M > N$ (so that, in particular, $M \geq 2$). For each $\sigma \coloneqq (\sigma_k)_{k \in \mathbb{N}} \in \Sigma_\infty$, define $\overline{\sigma} \coloneqq (\overline{\sigma}_n)_{n \in \mathbb{N}}$ by $\overline{\sigma}_{j(M+1)+l} \coloneqq \sigma_{jM+l}+ j$ for each $j \in \mathbb{N}_0$ and $l \in \{ 0, 1, \dotsc, M \}$, except when $(j,l) = (0,0)$. By the Euclidean algorithm, for each $n \in \mathbb{N}$, we can uniquely express $n = j(M+1)+l$ for some $j \in \mathbb{N}_0$ and $l \in \{ 0, 1, \dotsc, M \}$. Thus, the sequence $(\overline{\sigma}_n)_{n \in \mathbb{N}}$ is well-defined. We denote the map $\sigma \mapsto \overline{\sigma}$ on $\Sigma_\infty$ by $\Psi$.

\begin{claim} \label{definition of Psi}
We have $\Psi (\Sigma_\infty) \subseteq \Sigma_\infty$.
\end{claim}

\begin{proof}[Proof of Claim \ref{definition of Psi}] \renewcommand\qedsymbol{$\blacksquare$}
Let $\sigma \coloneqq (\sigma_k)_{k \in \mathbb{N}} \in \Sigma_\infty$. Put $\overline{\sigma} \coloneqq \Psi (\sigma)$, and write $\overline{\sigma} = (\overline{\sigma}_n)_{n \in \mathbb{N}}$. We need to show that $\overline{\sigma}_n < \overline{\sigma}_{n+1}$ for all $n \in \mathbb{N}$. For any $n \in \mathbb{N}$, it is clear that $p(M+1) \leq n \leq (p+1)(M+1)-1 = p(M+1) + M$ for some $p \in \mathbb{N}_0$. 

Fix $n \in \mathbb{N}$. First, assume that $p(M+1) \leq n < p(M+1) + M$ for some $p \in \mathbb{N}_0$. Then, $n = p(M+1)+l$ for some $l \in \{ 0, 1, \dotsc, M-1 \}$, and $n+1 = p(M+1) + (l+1)$ with $l+1 \in \{ 1, 2, \dotsc, M \}$. It follows that
\[
\overline{\sigma}_n = \sigma_{pM+l} + p < \sigma_{pM+(l+1)} + p = \overline{\sigma}_{n+1}.
\]
Now, assume $n = p(M+1) + M = (p+1)M + p$ for some $p \in \mathbb{N}_0$. Then, $n+1 = (p+1)(M+1)+0$, and we have the following:
\[
\overline{\sigma}_n = \sigma_{(p+1)M} + p < \sigma_{(p+1)M} + (p+1) = \overline{\sigma}_{n+1}.
\]
Thus, $\Psi (\sigma) = \overline{\sigma}$ is indeed in $\Sigma_\infty$, and this completes the proof of the claim. 
\end{proof}

Since $\Psi (f( \mathbb{I})) \subseteq f(\mathbb{I})$ by Proposition \ref{identity lemma} and Claim \ref{definition of Psi}, for any $x \in \mathbb{I}$, we infer that $(\Psi \circ f)(x) = f(\overline{x})$ for some $\overline{x} \in \mathbb{I}$, in which case the Pierce expansion digit sequence of $\overline{x}$ equals $(\Psi \circ f)(x)$. Note that $\overline{x} = (\varphi \circ f)(\overline{x}) = (\varphi \circ \Psi \circ f)(x)$ due to Proposition \ref{identity lemma}, and that the Pierce expansion of $\overline{x}$ is given by the following:
\begin{align} \label{definition of xbar}
\begin{aligned}
&\begin{array}{ccccccccccccc}
\overline{x} = & [ & d_1, & d_2, & \dotsc, & d_{M-1}, & d_M, & d_M+1, \\
 &  &  d_{M+1}+1, & d_{M+2}+1, & \dotsc, & d_{2M-1} + 1, & d_{2M} + 1, & d_{2M}+2, \\
 &  & d_{2M+1} + 2, & d_{2M+2} + 2, & \dotsc, & d_{3M-1}+2, & d_{3M} + 2, & d_{3M}+3, \\
 &  & & & \dotsc, & \\
 &  & d_{jM+1} + j, & d_{jM+2} + j & \dotsc, & d_{(j+1)M-1}+j, & d_{(j+1)M} + j, & d_{(j+1)M} + j+1, \\
 &  & & & \dotsc, & & & & ]_P,
\end{array}
\end{aligned}
\end{align}
where $d_k \coloneqq d_k(x)$ for all $k \in \mathbb{N}$.

Define $g \colon A \to \mathbb{I}$ by $g(x) \coloneqq (\varphi \circ \Psi \circ f)(x)$ for each $x \in A$. One can easily see that $g$ is well-defined and injective.

\begin{claim} \label{violation of xbar}
We have $g(A) \subseteq E$.
\end{claim}

\begin{proof} [Proof of Claim \ref{violation of xbar}] \renewcommand\qedsymbol{$\blacksquare$}
Let $x \in A$, and put $\overline{x} \coloneqq g(x)$. By the Euclidean algorithm, for each $n \in \mathbb{N}$, we can uniquely express $n = j(M+1)+l$ for some $j \in \mathbb{N}_0$ and $l \in \{ 0, 1, \dotsc, M \}$, so that $d_n(\overline{x}) = d_{jM+l}(x) + j$ by~\eqref{definition of xbar}. Note that $j \leq jM+l \leq d_{jM+l}(x)$ by Proposition \ref{digit condition lemma}(\ref{digit condition lemma 2}). Then, for any $n \in \mathbb{N}$, we have the following:
\[
d_{jM+l}(x) \leq d_n(\overline{x}) = d_{jM+l}(x) + j \leq 2 d_{jM+l}(x).
\]
By the definition of $A$, we know that the sequence $( (d_{jM+l}(x))^{1/(jM+l)} )_{j \in \mathbb{N}}$ converges to $e$ as $j \to \infty$ for any $l \in \{ 0, 1, \dotsc, M \}$, as it is a subsequence of $( (d_n (x))^{1/n} )_{n \in \mathbb{N}}$. Hence, for any $l \in \{ 0, 1, \dotsc, M \}$, we have
\begin{align*}
\liminf_{n \to \infty} (d_n (\overline{x}))^{\frac{1}{n}} 
&\leq \liminf_{j \to \infty} (d_{jM+l}(x)+j)^{\frac{1}{j(M+1)+l}} \\
&\leq \lim_{j \to \infty} \left[ (2 d_{jM+l}(x) )^{\frac{1}{jM+l}} \right]^{\frac{jM+l}{j(M+1)+l}} = e^{\frac{M}{M+1}}
\end{align*}
and
\begin{align*}
\limsup_{n \to \infty} (d_n (\overline{x}))^{\frac{1}{n}} 
&\geq \limsup_{j \to \infty} (d_{jM+l}(x)+j)^{\frac{1}{j(M+1)+l}} \\
&\geq \lim_{j \to \infty} \left[ (d_{jM+l}(x))^{\frac{1}{jM+l}}  \right]^{\frac{jM+l}{j(M+1)+l}} = e^{\frac{M}{M+1}}.
\end{align*}
Now, using the well-known facts that for any real-valued sequence $(a_n)_{n \in \mathbb{N}}$, we have $\liminf a_n \leq \limsup a_n$, and that $\lim a_n$ exists if and only if $\liminf a_n = \limsup a_n$, we conclude that $( d_n(\overline{x}) )^{1/n} \to e^{M/(M+1)} \neq e$ as $n \to \infty$. This proves the claim.
\end{proof}

For any $x, y \in \mathbb{I}$, define the following:
\[
\rho (x,y) \coloneqq \inf \{ k \in \mathbb{N} : (f(x))^{(k)} \neq (f(y))^{(k)} \} = \inf \{ k \in \mathbb{N} : d_k(x) \neq d_k(y) \}.
\]
Put $\sigma = (\sigma_j)_{j \in \mathbb{N}} \coloneqq f(x)$ and $\tau = (\tau_j)_{j \in \mathbb{N}} \coloneqq f(y)$. If $x \neq y$, then $\sigma \neq \tau$ since $f$ is injective by Proposition~\ref{identity lemma}, so there exists some $k \in \mathbb{N}$ such that $\sigma^{(k-1)} = \tau^{(k-1)}$ and $\sigma_k \neq \tau_k$. In this case, since the set over which the infimum is taken is a non-empty subset of positive integers, it has the smallest element. Therefore, $\rho$ is well-defined.

We will derive both the lower and upper bounds for the distance between two distinct irrational numbers $x$ and $y$ in terms of their Pierce expansion digits.

\begin{claim} \label{bounds for distance claim}
Let $x, y \in \mathbb{I}$ with $x \neq y$, and put $n \coloneqq \rho (x,y)$, which is finite. If $d_n(x) < d_n(y)$, then
\begin{align} \label{bounds for distance}
\left( \prod_{k=1}^{n-1} \frac{1}{d_k} \right) \frac{1}{d_n(y) (d_{n+1}(y)+1)} \leq |x-y| \leq \prod_{k=1}^{n-1} \frac{1}{d_k},
\end{align}
where $d_k \coloneqq d_k(x) = d_k(y)$ for each $k \in \{ 1, 2, \dotsc, n-1 \}$, with the convention that the product over the empty set equals $1$.
\end{claim}

\begin{proof} [Proof of Claim \ref{bounds for distance claim}] \renewcommand\qedsymbol{$\blacksquare$}
By the definition of $n$, we have $d_k(x) = d_k(y)$ for each $k \in \{ 1, 2, \dotsc, n-1 \}$, and we can denote the common value by $d_k$.

Put $\sigma \coloneqq f(x)$ and $\tau \coloneqq f(y)$. Then, by Proposition \ref{identity lemma}, both $\sigma$ and $\tau$ are elements of $\Sigma_\infty$. Furthermore, observe that $\tau^{(n)} \in \Sigma_n$ and $\tau^{(n+1)} \in \Sigma_{n+1}$. By the definition of $I(\sigma)$, we have $y \in I ({\tau^{(n+1)}}) \subseteq I ({\tau^{(n)}})$, though $x \not \in I ({\tau^{(n)}})$. Now, consider all the endpoints of $I ({\tau^{(n)}})$ and $I ({\tau^{(n+1)}})$. Exactly two of the endpoints, $\varphi (\tau^{(n)})$ and $\varphi (\widehat{\tau^{(n+1)}})$, lie between $x$ and $y$, while the other two, $\varphi (\widehat{\tau^{(n)}})$ and $\varphi (\tau^{(n+1)})$, do not. Indeed, Proposition \ref{I sigma} and the assumption $d_n(x)<d_n(y)$ tell us that 
\[
\begin{cases}
\varphi (\widehat{\tau^{(n)}}) < \varphi (\tau^{(n+1)}) < y < \varphi (\widehat{\tau^{(n+1)}}) < \varphi ({\tau^{(n)}}) \leq \varphi (\widehat{\sigma^{(n)}}) < x, &\text{if $n$ is odd}; \\
\varphi (\widehat{\tau^{(n)}}) > \varphi (\tau^{(n+1)}) > y > \varphi (\widehat{\tau^{(n+1)}}) > \varphi ({\tau^{(n)}}) \geq \varphi (\widehat{\sigma^{(n)}}) > x, &\text{if $n$ is even}.
\end{cases}
\]
Since $\tau^{(n)}$ and $\widehat{\tau^{(n+1)}}$ share the initial $n$ terms, $(d_k(y))_{k=1}^n$, we can apply the definition of $\varphi$ to obtain the following:
\begin{align*}
|x-y|
&\geq |\varphi (\tau^{(n)}) - \varphi (\widehat{\tau^{(n+1)}})| = \left( \prod_{k=1}^{n-1} \frac{1}{d_k} \right) \frac{1}{d_n(y) (d_{n+1}(y)+1)}.
\end{align*}
On the other hand, since $\varphi (\sigma) = (\varphi \circ f) (x) = x$ and $\varphi (\tau) = (\varphi \circ f) (y) = y$ by Proposition \ref{identity lemma}, and since $\sigma^{(n-1)} = \tau^{(n-1)}$ by the definition of $n$, we have the following:
\[
|x-y|
= |[\varphi (\sigma) - \varphi (\sigma^{(n-1)})] - [\varphi (\tau) - \varphi (\tau^{(n-1)})]|.
\]
This gives us the following inequality:
\[
|x-y| \leq |\varphi (\sigma) - \varphi (\sigma^{(n-1)})| + |\varphi (\tau) - \varphi (\tau^{(n-1)})|.
\]
Using Proposition \ref{distance between sigma and sigma n}, we obtain
\[
|x-y| \leq \left( \prod_{k=1}^{n-1} \frac{1}{d_k} \right) \left( \frac{1}{d_n(x)} + \frac{1}{d_n(y)} \right) < \prod_{k=1}^{n-1} \frac{1}{d_k},
\]
as desired.
\end{proof}

Recall that $M>N$ and, in particular, $M \geq 2$. Let
\begin{align*} 
\varepsilon \coloneqq \frac{12}{M} \quad \text{and} \quad c \coloneqq 4^{6M(M+1)}.
\end{align*}
For later use, we note that
\begin{align} \label{bound of epsilon and c}
\varepsilon > \frac{1}{M \log M-1} \quad \text{and} \quad c > 2e \cdot 4^{M(M+1)}.
\end{align}
Indeed, a direct calculation shows that the former inequality holds when $M=2$. For $M \geq 3$, we have $M \log M > 2$ and $\log M > 1$ so that $(M \log M - 1)^{-1} < 6 (M \log M-1)^{-1} < 12 (M \log M)^{-1} < 12M^{-1} = \varepsilon$. As for the latter inequality, it suffices to observe that $4^{5M(M+1)} > 4^{2 \cdot 3} > 2e$.

\begin{claim} \label{Holder condition}
For any $x,y \in A$, we have the following:
\[
| g (x) - g (y) | \geq c^{-1} | x - y |^{1+3 \varepsilon}.
\]
\end{claim}

\begin{proof} [Proof of Claim \ref{Holder condition}] \renewcommand\qedsymbol{$\blacksquare$}
If $x=y$, the inequality trivially holds. Suppose $x \neq y$, and let $n \coloneqq \rho (x,y) < +\infty$. Then, $d_k(x) = d_k(y)$ for each $k \in \{ 1, 2, \dotsc, n-1 \}$, and we denote the common value by $d_k$. Without loss of generality, assume that $d_n(x) < d_n(y)$. Put $\overline{x} \coloneqq g(x)$ and $\overline{y} \coloneqq g(y)$. By the injectivity of $g$, it follows that $\overline{x} \neq \overline{y}$. Put $r \coloneqq \rho (\overline{x}, \overline{y}) < +\infty$. Then, $d_k(\overline{x}) = d_k(\overline{y})$ for each $k \in \{ 1, 2, \dotsc, r-1 \}$, and it is clear that $d_r(\overline{x}) < d_r(\overline{y})$. Now, we break the proof into two cases depending on the size of $n$.

{\sc Case I.}
Assume $n \leq 2M$. Then, by \eqref{definition of xbar}, we have $r \leq 2M+1$, which implies that $d_{r+1}(\overline{y}) \leq d_{2M+2} (\overline{y})$ by Proposition \ref{digit condition lemma}(\ref{digit condition lemma 1}). Using the lower bound from \eqref{bounds for distance}, along with \eqref{definition of xbar}, we obtain the following:
\begin{align*}
|\overline{x}-\overline{y}|
&\geq \left( \prod_{k=1}^r \frac{1}{d_k(\overline{y})} \right) \frac{1}{d_{r+1}(\overline{y})+1} \\
&\geq \left( \prod_{k=1}^{2M+1} \frac{1}{d_k(\overline{y})} \right) \frac{1}{d_{2M+2}(\overline{y})+1} 
= \left( \prod_{k=1}^{M} \frac{1}{d_k(y)} \right) \frac{1}{d_M(y)+1} \left( \prod_{k=M+1}^{2M} \frac{1}{d_k(y)+1} \right) \frac{1}{d_{2M}(y)+3}.
\end{align*}
Since $y \in A$ and $M>N$, Proposition \ref{digit condition lemma}(\ref{digit condition lemma 1}) and \eqref{bounds for dn} tell us that $d_1(y)<d_2(y) < \dotsb < d_M(y) < 4^M$ and $d_k(y) < 4^k$ for each $k \in \{ M+1, M+2, \dotsc, 2M \}$. Thus, we have the following:
\begin{align*}
|\overline{x}-\overline{y}|
&\geq \left( \prod_{k=1}^M \frac{1}{4^M} \right) \frac{1}{4^M+1} \left( \prod_{k=M+1}^{2M} \frac{1}{4^k+1} \right) \frac{1}{4^{2M}+3} \\
&> \left( \frac{1}{4^{2M}+3} \right)^{2M+2} > \left( \frac{1}{4^{3M}} \right)^{2M+2} = \frac{1}{4^{6M(M+1)}} = c^{-1}.
\end{align*}
Therefore, since $|x-y| \leq 1$ and $1+3\varepsilon > 1$, it follows that $|\overline{x}-\overline{y}| \geq c^{-1} |x-y|^{1+3\varepsilon}$.

{\sc Case II.}
Assume $pM < n \leq (p+1)M$ for some positive integer $p \geq 2$. Then, by \eqref{definition of xbar}, we have $p(M+1) +1 \leq r \leq p(M+1)+M$. To obtain a lower bound for $|\overline{x}-\overline{y}|$ using \eqref{bounds for distance}, we first note that $d_{r+1}(\overline{y}) \leq d_{n+1}(y) + p$ by considering the following two cases:
\begin{itemize}
\item
{\sc Case 1:} If $pM<n<(p+1)M$, then
\[
d_r (\overline{y}) = d_n(y) + p \quad \text{and} \quad d_{r+1}(\overline{y}) = d_{n+1}(y) + p.
\]
\item
{\sc Case 2:} If $n=(p+1)M$, then, since $d_n(y) + 1 \leq d_{n+1} (y)$ by Proposition \ref{digit condition lemma}(\ref{digit condition lemma 1}), we have
\[
d_r (\overline{y}) = d_{n}(y) + p \quad \text{and} \quad d_{r+1}(\overline{y}) = d_{n} (y)+ (p+1)  \leq d_{n+1}(y) + p.
\]
(See the two rightmost columns in \eqref{definition of xbar}.)
\end{itemize}
Thus, by the lower bound in \eqref{bounds for distance}, \eqref{definition of xbar}, and the bound for $d_{r+1}(\overline{y})+1$ obtained above, we have the following:
\begin{align*}
|\overline{x}-\overline{y}|
&\geq \left( \prod_{k=1}^r \frac{1}{d_k(\overline{y})} \right) \frac{1}{d_{r+1}(\overline{y})+1}
= \left( \prod_{k=1}^{p(M+1)} \frac{1}{d_k(\overline{y})} \right) \left( \prod_{k=p(M+1)+1}^r \frac{1}{d_k(\overline{y})} \right) \frac{1}{d_{r+1}(\overline{y})+1}.
\end{align*}
Breaking this into products over different ranges, we have the following:
\begin{align*}
|\overline{x}-\overline{y}|
&\geq \left[ \left( \prod_{k=1}^M \frac{1}{d_k(y)} \right) \frac{1}{d_M(y)+1} \right]
\left[ \left( \prod_{k=M+1}^{2M} \frac{1}{d_k(y)+1} \right) \frac{1}{d_{2M}(y)+2} \right]  \\
&\quad \times \dotsb \times \left[ \left( \prod_{k=(p-1)M+1}^{pM} \frac{1}{d_k(y)+(p-1)} \right) \frac{1}{d_{pM}(y)+p} \right]
\left( \prod_{k=pM+1}^n \frac{1}{d_k(y)+p} \right) \frac{1}{d_{n+1}(y) + (p+1)}.
\end{align*}
This simplifies to the following:
\begin{align} \label{first estimation}
|\overline{x}-\overline{y}|
\geq \left( \prod_{j=0}^{p-1} \left[ \left( \prod_{l=1}^M \frac{1}{d_{jM+l} +j} \right) \frac{1}{d_{(j+1)M}+(j+1)} \right] \right) \left( \prod_{k = pM+1}^{n-1} \frac{1}{d_k+p} \right) \frac{1}{(d_n(y)+p)(d_{n+1}(y) + p+1)}.
\end{align}

Now, we investigate the term inside the square brackets in \eqref{first estimation} by considering two cases depending on the value of $j$. For $j =0$, we have $d_1(y) < d_2(y) < \dotsb < d_M(y) < 4^M$ by Proposition \ref{digit condition lemma}(\ref{digit condition lemma 1}) and \eqref{bounds for dn}; therefore,
\begin{align} \label{second square bracket estimation}
\left( \prod_{l=1}^M \frac{1}{d_l(y)} \right) \frac{1}{d_M(y)+1} 
> \left( \prod_{l=1}^M \frac{1}{4^M} \right) \frac{1}{4^M+1} > \frac{1}{4^{M \cdot M}} \cdot \frac{1}{4^M+4^M} = \frac{1}{2 \cdot 4^{M(M+1)}}.
\end{align}

Fix $j \in \{ 1, 2, \dotsc, p-1 \}$. Since $y \in A$ and $(j+1)M > M$, by \eqref{bounds for dn}, we have the following:
\begin{align} \label{eq10}
d_{(j+1)M}(y) + (j+1) 
&< 4^{(j+1)M} + (j+1) \nonumber \\
&< 4^{(j+1)M} + 4^{(j+1)M} = 2^{2jM+(2M+1)} \nonumber \\
&< 2^{2jM + 4jM} = 4^{3jM}.
\end{align}
Now, note that
\begin{align*}
\varepsilon = \frac{12}{M} 
&> \frac{12}{2M + \dfrac{M+1}{j}} 
= \frac{6jM}{jM^2 + \dfrac{M(M+1)}{2}}
= \frac{3jM}{\sum \limits_{l=1}^M (jM+l)} \cdot \frac{\log 4}{\log 2} 
= \frac{3jM \log 4}{\sum \limits_{l=1}^M \log (2^{jM+l})} \\
&> \frac{3jM \log 4}{\sum \limits_{l=1}^M \log (2^{jM+l}+j)} 
= \frac{3jM \log 4}{\log \left( \prod \limits_{l=1}^M (2^{jM+l}+j) \right)},
\end{align*}
which leads to
\begin{align} \label{eq11}
3jM \log 4 <  \varepsilon \log \left( \prod_{l=1}^M (2^{jM+l}+j) \right),
\quad \text{or, equivalently,} \quad
4^{3jM} <  \left( \prod_{l=1}^M (2^{jM+l}+j) \right)^\varepsilon.
\end{align}
Hence, from \eqref{eq10} and \eqref{eq11}, we obtain the following:
\[
d_{(j+1)M}(y) + (j+1) 
< \left( \prod_{l=1}^M (2^{jM+l}+j) \right)^{\varepsilon} 
< \left( \prod_{l=1}^M (d_{jM+l}(y)+j) \right)^{\varepsilon}.
\]
Thus, we obtain the following lower bound for the square-bracketed term in \eqref{first estimation} for each $j \in \{ 1, 2, \dotsc, p-1 \}$:
\begin{align} \label{first square bracket estimation}
\left( \prod_{l=1}^M \frac{1}{d_{jM+l} +j} \right) \frac{1}{d_{(j+1)M}+(j+1)} > \left( \prod_{l=1}^M \frac{1}{d_{jM+l}+j} \right)^{1+\varepsilon}.
\end{align}

Next, we find a lower bound for the rightmost term $[(d_n(y)+p) (d_{n+1}(y) + p+1)]^{-1}$ in \eqref{first estimation}. By \eqref{bounds for dn}, we have the following:
\[
(d_n(y)+p) (d_{n+1}(y) + p+1) < (4^n+p)(4^{n+1}+p+1).
\]
Since $pM<n$, it is clear that $p<n$, and thus $p < 4^n$ and $p+1 < 4^{n+1}$. Then,
\begin{align} \label{eq13}
(d_n(y)+p) (d_{n+1}(y) + p+1) < (2 \cdot 4^n)(2 \cdot 4^{n+1}) = 4^{2n+2} \leq 4^{3n}.
\end{align}
Now, note that since $p \geq 2$ and $pM+1 \leq n \leq (p+1)M$, we have
\begin{align*}
\varepsilon = \frac{12}{M}
&> \frac{12}{(p-1)M+1}
= \frac{6(M+pM)}{\dfrac{M+pM}{2} \cdot (pM-(M-1))}
= \frac{3(p+1)M}{\sum \limits_{k=M}^{pM} k} \cdot \frac{\log 4}{\log 2} \\
&\geq \frac{3n}{\sum \limits_{k=M}^{n-1} k} \cdot \frac{\log 4}{\log 2}
= \frac{3n \log 4}{\log \left( \prod \limits_{k=M}^{n-1} 2^k \right)},
\end{align*}
which leads to
\begin{align} \label{eq14}
3n \log 4 \leq \varepsilon \log \left( \prod_{k=M}^{n-1} 2^k \right),
\quad \text{or, equivalently,} \quad
4^{3n} \leq \left( \prod_{k=M}^{n-1} 2^k \right)^\varepsilon.
\end{align}
Hence, from \eqref{eq13} and \eqref{eq14}, together with \eqref{bounds for dn}, we obtain
\[
(d_n(y)+p) (d_{n+1}(y) + p+1)
\leq \left( \prod_{k=M}^{n-1} 2^k \right)^{\varepsilon} \leq \left( \prod_{k=M}^{n-1} d_k(y) \right)^\varepsilon,
\]
which implies
\begin{align} \label{last two terms estimation}
\frac{1}{(d_n(y)+p) (d_{n+1}(y) + p+1)} \geq \left( \prod_{k=M}^{n-1} \frac{1}{d_k} \right)^\varepsilon \geq \left( \prod_{k=1}^{n-1} \frac{1}{d_k} \right)^\varepsilon.
\end{align}
Combining \eqref{first estimation}, \eqref{second square bracket estimation}, \eqref{first square bracket estimation}, and \eqref{last two terms estimation}, we now have the following:
\begin{align*}
|\overline{x}-\overline{y}| \geq \frac{1}{2 \cdot 4^{M(M+1)}} \left[ \prod_{j=1}^{p-1} \left( \prod_{l=1}^M \frac{1}{d_{jM+l} +j} \right)^{1+\varepsilon} \right] \left( \prod_{k=pM+1}^{n-1} \frac{1}{d_k +p} \right) \left( \prod_{k=1}^{n-1} \frac{1}{d_k} \right)^\varepsilon.
\end{align*}
By Proposition \ref{digit condition lemma}(\ref{digit condition lemma 2}), for all $j \in \{ 1, 2, \dotsc, p-1 \}$ and $l \in \{ 1, 2, \dotsc, M \}$, we have $jM < jM+l \leq d_{jM+l}$, which implies
\[
d_{jM+l} + j < d_{jM+l} + \frac{1}{M} d_{jM+l} = \left( 1 + \frac{1}{M} \right) d_{jM+l}.
\]
Similarly, for all $k \in \{ pM+1, pM+2, \dotsc, n-1 \}$, we have $pM < k \leq d_k$, which implies $d_k  + p < ( 1 + M^{-1} ) d_k$. Thus,
\begin{align} \label{second estimation}
|\overline{x}-\overline{y}|
\geq\ & \frac{1}{2 \cdot 4^{M(M+1)}} \left[ \prod_{j=1}^{p-1} \left( \prod_{l=1}^M \frac{1}{( 1 + M^{-1} ) d_{jM+l}} \right)^{1+\varepsilon} \right] \left( \prod_{k=pM+1}^{n-1} \frac{1}{( 1 + M^{-1} ) d_k} \right) \left( \prod_{k=1}^{n-1} \frac{1}{d_k} \right)^\varepsilon \nonumber \\
=\ & \frac{1}{2 \cdot 4^{M(M+1)}} \left[ \prod_{j=1}^{p-1} \left( \frac{1}{(1 + M^{-1} )^M} \prod_{k=jM+1}^{(j+1)M} \frac{1}{d_{k}} \right)^{1+\varepsilon} \right] \frac{1}{( 1 + M^{-1} )^{n-1-pM}} \nonumber \\
& \times \left( \prod_{k=pM+1}^{n-1} \frac{1}{d_k} \right) \left[ \left( \prod_{k=1}^{M} \frac{1}{d_k} \right)^\varepsilon \left( \prod_{k=M+1}^{pM} \frac{1}{d_k} \right)^\varepsilon \left( \prod_{k=pM+1}^{n-1} \frac{1}{d_k} \right)^\varepsilon \right] \nonumber \\
\geq\ & \frac{1}{2 \cdot 4^{M(M+1)}} \frac{1}{( 1 + M^{-1} )^M} \left( \prod_{k=1}^M \frac{1}{d_k} \right)^\varepsilon \nonumber \\
& \times \left[ \prod_{j=1}^{p-1} \left( \frac{1}{( 1 + M^{-1} )^{M(1+\varepsilon)}} \left( \prod_{k=jM+1}^{(j+1)M} \frac{1}{d_k} \right)^{1+2 \varepsilon} \right) \right] \left( \prod_{k=pM+1}^{n-1} \frac{1}{d_k} \right)^{1+\varepsilon}.
\end{align}
The last inequality holds since $n-1 < (p+1)M$ by the assumption for {\sc Case II}.

Next, we estimate each term of the product in the square-bracketed part in \eqref{second estimation}. Fix $j \in \{ 1, 2, \dotsc, p-1 \}$. Recall from \eqref{bound of epsilon and c} that $\varepsilon > (M \log M-1)^{-1}$. Since $k \leq d_k$ for all $k \in \mathbb{N}$ by Proposition~\ref{digit condition lemma}(\ref{digit condition lemma 2}), we have the following:
\begin{align*}
\varepsilon
&> \frac{1}{M \log M-1} = \frac{1}{\sum \limits_{l=1}^M \log M - 1} 
> \frac{1}{\sum \limits_{l=1}^M \log (jM+l) - 1} \\
&\geq \frac{1}{\sum \limits_{l=1}^M \log d_{jM+l} - 1}
= \frac{1}{\sum \limits_{k=jM+1}^{(j+1)M} \log d_k - 1}
= \frac{1}{\log \left( \prod \limits_{k=jM+1}^{(j+1)M} d_{k} \right) - 1},
\end{align*}
which implies
\begin{align*}
1 + \varepsilon \leq \varepsilon \log \left( \prod_{k=jM+1}^{(j+1)M} d_{k} \right),
\quad \text{or, equivalently,} \quad
e^{1+\varepsilon} \leq \left( \prod_{k=jM+1}^{(j+1)M} d_{k} \right)^\varepsilon.
\end{align*}
Using the well-known fact that
\begin{align} \label{e is bounded}
\left( 1 + \frac{1}{n} \right)^n \leq e \quad \text{for all } n \in \mathbb{N},
\end{align}
we deduce that
\[
\left( 1 + \frac{1}{M} \right)^{M(1+\varepsilon)} \leq \left( \prod_{k=jM+1}^{(j+1)M} d_k \right)^\varepsilon.
\]
Thus, for any $j \in \{ 1, 2, \dotsc, p-1 \}$, we have
\begin{align} \label{estimation of 1+1/M}
\frac{1}{\left( 1 + M^{-1} \right)^{M(1+\varepsilon)}} \left( \prod_{k=jM+1}^{(j+1)M} \frac{1}{d_k} \right)^{1+2\varepsilon}
\geq \left( \prod_{k=jM+1}^{(j+1)M} \frac{1}{d_k} \right)^{1 + 3 \varepsilon}.
\end{align}
Therefore, combining \eqref{second estimation} and \eqref{estimation of 1+1/M}, we finally obtain the following:
\begin{align*}
|\overline{x}-\overline{y}|
&\geq \frac{1}{2 \cdot 4^{M(M+1)}} \frac{1}{\left( 1 + M^{-1} \right)^M} \left( \prod_{k=1}^M \frac{1}{d_k} \right)^\varepsilon \left( \prod_{k=M+1}^{pM} \frac{1}{d_k} \right)^{1+3\varepsilon} \left( \prod_{k=pM+1}^{n-1} \frac{1}{d_k} \right)^{1+\varepsilon} \\
&\geq \frac{1}{2e \cdot 4^{M(M+1)}} \left( \prod_{k=1}^{n-1} \frac{1}{d_k(y)} \right)^{1+3 \varepsilon}
\geq c^{-1} |x-y|^{1+3\varepsilon},
\end{align*}
where the second inequality follows from \eqref{e is bounded}, and the last from \eqref{bound of epsilon and c} and the upper bound in \eqref{bounds for distance}.
\end{proof}

For the map $g \colon A \to \mathbb{I}$, let its inverse on $g(A)$ be denoted by $g^{-1} \colon g(A) \to A$. By Claim \ref{dimension of A}, we have~$1 = \hdim A = \hdim g^{-1}(g(A))$. Additionally, by Claim \ref{Holder condition}, for all $\overline{x}, \overline{y} \in g(A)$, we have the inequality
\[
|g^{-1} (\overline{x}) - g^{-1}(\overline{y})| \leq c^{1/(1+3\varepsilon)} |\overline{x} - \overline{y}|^{1/(1+3\varepsilon)},
\]
which implies that $g^{-1}$ is $[1/(1+3\varepsilon)]$-H{\"o}lder continuous. Consequently, by Proposition \ref{Holder dimension}, we obtain $\hdim g^{-1}(g(A)) \leq (1+3\varepsilon) \hdim g(A)$, and thus $\hdim \allowbreak g(A) \allowbreak \geq 1/(1+3\varepsilon)$. By applying the monotonicity property of the Hausdorff dimension (see \cite[p.~48]{Fal14}) to the inclusion $E \supseteq g(A)$ (Claim~\ref{violation of xbar}), we obtain the following:
\[
\hdim E \geq \hdim g(A) \geq \frac{1}{1+3\varepsilon} = \frac{1}{1 + (36/M)}.
\]
Since $M$ ($>N$) can be chosen arbitrarily large, we conclude that $\hdim E \geq 1$, as required.
\end{proof}

\begin{remark} \label{digit condition remark}
Our definition of $\overline{x}$ in \eqref{definition of xbar} differs from the one in Wu's original work \cite{Wu00} on the LLN in Engel expansions. This discrepancy was inevitable because the expansion constructed in \cite{Wu00} only satisfies the non-decreasing condition for the digits, but not the strict increasing condition; therefore, it is not Pierce realizable according to Proposition \ref{strict increasing condition}. In other words, the expression considered in \cite{Wu00} cannot represent a Pierce expansion of any number. However, a slight modification of the approach allowed us to follow a similar argument. We further note that Proposition \ref{digit condition lemma}(\ref{digit condition lemma 2}), which does not hold for Engel expansions, played a crucial role in our proof.
\end{remark}

\section{Concluding remarks}

\subsection{Future directions}

We have presented an elementary proof which demonstrated that the set of exceptions to the law of large numbers in Pierce expansions has full Hausdorff dimension. By adapting Wu's method, which was initially developed for Engel expansions, to the context of Pierce expansions, we have bridged gaps in previous proofs in \cite{Ahn24} and provided a more accessible perspective. This work highlights the fractal nature of exceptional sets and opens the door for further investigations in related areas.

This elementary approach underscores the universality of fractal dimensions in number theory and dynamical systems, showing how modifications to established methods can yield simpler, yet effective, proofs for complex results. Moreover, it opens the possibility of extending similar analyses to other exceptional sets in Pierce expansions, or even to other real number representation systems, such as Engel, L{\"u}roth, and Sylvester expansions. From an educational standpoint, the approach is valuable, and offers students a tangible gateway to advanced concepts through concrete examples. For instance, the sets
\begin{align} \label{sets in conclusion}
\left\{ x \in (0,1] : \lim_{n \to \infty} \frac{d_{n+1}(x)}{d_n(x)} = \alpha \right\} \quad \text{and} \quad \left\{ x \in (0,1] : \lim_{n \to \infty} \frac{\log d_{n+1}(x)}{\log d_n(x)} = \beta \right\},
\end{align}
which respectively concern the ratio of two consecutive Pierce expansion digits for $\alpha \in [1, +\infty]$ and the ratio of the logarithms of these digits for $\beta \in [1, +\infty]$, are shown to have full Hausdorff dimension~\cite[Theorem 1.8]{Ahn24} and Hausdorff dimension $1/\beta$ \cite[Theorem 1.12]{Ahn24}, respectively. In this paper, the straightforward definition of the sets, such as those in \eqref{sets in conclusion}, allowed us to provide an elementary proof without relying on more advanced techniques such as symbolic dynamics and the measure-theoretic dimension theory. As demonstrated, for sets similar to \eqref{sets in conclusion}, the simplicity of their definitions enables the application of this elementary approach, where particular sequences with different growth rates are constructed.

Future research could explore the use of specific sequences, as presented in this paper, to bridge the gap between specialists in fractal geometry and the broader mathematical community. By utilizing such sequences, complex results could be made more accessible to general mathematicians, thus potentially opening new avenues for research in number theory and fractal geometry.

\subsection{Limitations and scope}

While this elementary proof provides a more approachable framework, it is worth noting that it may not fully capture the generality of more advanced techniques, such as symbolic dynamics. For example, in \cite{Ahn25}, we showed that various exceptional sets related to Shallit's law of leap years in Pierce expansions have full Hausdorff dimension; the law itself was discussed in \cite{Sha93}. This work introduced a dynamical concept called {\em finite replacement-invariance}. Specifically, the sets studied in \cite{Ahn25} are
\[
S(\alpha) \coloneqq \left\{ x \in [0,1] : \limsup_{N \to \infty} s_N(x) = \frac{1}{\sqrt{2\alpha}} \text{ and } \liminf_{N \to \infty} s_N(x) = - \frac{1}{\sqrt{2\alpha}} \right\}, \quad \alpha \in [0, +\infty],
\]
where
\[
s_N(x) \coloneqq \dfrac{\displaystyle N \sum_{k \in \mathbb{N}} \frac{(-1)^{k+1}}{d_1(x) \dotsm d_k(x)} - \sum_{k \in \mathbb{N}} (-1)^{k+1} \left\lfloor \dfrac{N}{d_1(x) \dotsm d_k(x)} \right\rfloor}{\sqrt{\log N}}, \quad x \in [0,1].
\]
Recall that we constructed a specific sequence with a slower growth rate than typical sequences (in the Lebesgue measure sense) in this paper. However, it is not immediately clear what the quantity $s_N(x)$ represents or what kind of sequences satisfy the $\limsup$ and $\liminf$ conditions. Hence, it appears that our approach, which involves sequences with different convergence rates, may not be directly applicable to studying the sets $S(\alpha)$. Nevertheless, this does not entirely preclude such applications. Future research could explore how these methods might be integrated or extended to make advanced concepts in number theory and fractal geometry more accessible to a broader mathematical audience. By doing so, this line of inquiry holds promise for deepening our understanding of number theory, fractal geometry, and their interconnections.

\section*{Acknowledgments}

I would like to sincerely thank the State University of New York at Buffalo (USA), where part of this research was carried out during my doctoral studies. I am also grateful to the referees and the Editor for their insightful and constructive feedback, which significantly enhanced the quality of this paper. Their guidance has been instrumental in shaping the final version of this work.

\end{document}